\documentclass[10pt]{amsart}
\usepackage{amsmath, amsthm, amsfonts, amssymb, graphicx, hyperref, bm,verbatim,mathrsfs,siunitx}
\usepackage{stmaryrd,enumerate}
\theoremstyle{plain}
\newtheorem{thrm}{Theorem}[section]
\newtheorem{lmm}[thrm]{Lemma}

\numberwithin{sblmm}{thrm} 
\numberwithin{equation}{section}

\newcommand{\Mod}[1]{\ (\mathrm{mod}\ #1)}
\renewcommand{\mod}[1]{\mathrm{mod}\ #1}
\begin{document}
\title{Primes and polynomials with restricted digits}
\author{James Maynard}
\address{Magdalen College, Oxford, England, OX1 4AU}
\email{james.alexander.maynard@gmail.com}
\begin{abstract}
Let $q$ be a sufficiently large integer, and $a_0\in\{0,\dots,q-1\}$. We show there are infinitely many prime numbers which do not have the digit $a_0$ in their base $q$ expansion. Similar results are obtained for values of a polynomial (satisfying the necessary local conditions) and if multiple digits are excluded.
\end{abstract}
\maketitle
\section{Introduction}\label{sec:Introduction}
Let $a_0\in\{0,\dots,q-1\}$ and let 
\[\mathcal{A}=\Bigl\{\sum_{i\ge 0}n_i q^i: n_i\in\{0,\dots,q-1\}\backslash\{a_0\}\Bigr\}\]
be the set of numbers which have no digit equal to $a_0$ when written in base $q$. For fixed $q$, the number of elements of $\mathcal{A}$ which are less than $x$ is $O(x^{1-\epsilon_q})$, where $\epsilon_q=\log{(q/(q-1))}/\log{q}>0$. In particular, $\mathcal{A}$ is a sparse subset of the natural numbers. A set being sparse in this way presents several analytic difficulties if one tries to answer arithmetic questions such as whether the set contains infinitely many primes. Typically we can only show that sparse sets contain infinitely many primes when the set in question possesses some additional multiplicative structure.

The set $\mathcal{A}$ has unusually nice structure in that its Fourier transform has a convenient explicit analytic description, and is often unusually small in size. There has been much previous work \cite{1,2,3,4,5,6,7} studying $\mathcal{A}$ and related sets by exploiting this Fourier structure. In particular the work of Dartyge and Mauduit \cite{Almost1,Almost2} shows the existence of infinitely many integers in $\mathcal{A}$ with at most $2$ prime factors, this result relying on the fact that $\mathcal{A}$ is well distributed in arithmetic progressions \cite{Almost1,Level1,Level2}. We also mention the related work of Mauduit-Rivat \cite{MauduitRivat} who showed the sum of digits of primes was well-distributed, and the work of Bourgain \cite{Bourgain} which showed the existence of primes in the sparse set created by prescribing a positive proportion of the digits.

We show that there are infinitely many primes in $\mathcal{A}$, and any polynomial $P$ satisfying suitable local conditions takes infinitely many values in $\mathcal{A}$ provided the base $q$ is sufficiently large (i.e. provided $\mathcal{A}$ is not too sparse). Our proof is based on the circle method, and in particular makes key use of the Fourier structure of $\mathcal{A}$, in the same spirit as the aforementioned works. Somewhat surprisingly, the Fourier structure is sufficient to deduce the existence of primes in $\mathcal{A}$ using only existing exponential sum estimates for the primes, and without having to investigate further bilinear sums.
\begin{thrm}\label{thrm:Prime}
Let $q>\num{2000000}$, $a_0\in\{0,\dots,q-1\}$ and $\mathcal{A}=\{\sum_{i\ge 0}n_iq^i: n_i\in\{0,\dots,q-1\}\backslash\{a_0\}\}$ be the set of numbers with no digit  in base $q$ equal to $a_0$. Then for any constant $A>0$ we have
\[\sum_{n<q^k}\Lambda(n)\mathbf{1}_{\mathcal{A}}(n)=\kappa_q(a_0)(q-1)^k+O_A\Bigl(\frac{(q-1)^k}{(\log{q^k})^A}\Bigr),\]
where
\[\kappa_q(a_0)=\begin{cases}
\displaystyle\frac{q}{q-1}\,\qquad&\text{if $(a_0,q)\ne 1$,}\\
\displaystyle\frac{q(\phi(q)-1)}{(q-1)\phi(q)},\qquad&\text{if $(a_0,q)=1$.}
\end{cases}
\]
\end{thrm}
Thus there are infinitely many primes with no digit $a_0$ when written in base $q$. There is nothing special about the fact we sum up to a power of $q$; one could sum $n$ up to $x$ instead of $q^k$ and have $\sum_{n<x}\mathbf{1}_{\mathcal{A}}(n)$ instead of $(q-1)^k$ in the statement.

We have made no particular effort to optimize the lower bound on $q$; it is likely that it could be improved significantly. In particular, a more involved calculation shows that $q>\num{2500}$ is sufficient by the same method, whilst it appears that the method of bilinear sums, Harman's sieve and zero density estimates all have the potential to show the existence of primes missing digits when the base is noticeably smaller. One might conjecture that the result would remain true for all $q>2$.

As presented here the bound is ineffective due to the reliance on estimates for primes in arithmetic progressions. However, since these estimates are only used when the modulus is highly composite, in fact Siegel zeros do not play a role, and so the error terms could be replaced by effective ones of size $O((q-1)^k\exp(-c k^{1/2}))$ if desired.

An analysis of our method reveals that in fact one can choose digits $a_0,\dots,a_{k-1}\in\{0,\dots,q-1\}$, and we obtain the same statement for primes $p=\sum_{i=0}^{k-1}p_i q^i$ with $p_i\in\{0,\dots,q-1\}\backslash\{a_i\}$ uniformly over all such choices of $a_0,\dots a_{k-1}$.

Our results hold for $q$ sufficiently large not only because we require $\mathcal{A}$ to be not too sparse, but also because we separately get superior $L^1$ control on the Fourier transform of $\mathcal{A}$ as $q$. A similar feature was present in the earlier work \cite{MauduitPolys}.
\begin{thrm}\label{thrm:Poly}
Let $q>\exp(\exp(2r))$, and $P\in\mathbb{Z}[X]$ be a polynomial of degree $r$ with lead coefficient $a_r$. Then for any $A>0$ we have
\[\sum_{P(n)<q^k}\mathbf{1}_{\mathcal{A}}(P(n))=a_r^{1/r}\mathfrak{S}(P)\frac{q^{k/r}(q-1)^k}{q^k}+O_{P,A}\Bigl(\frac{q^{k/r}(q-1)^k}{q^k(\log{q^k})^A}\Bigr),\]
where 
\[\mathfrak{S}(P)=\lim_{J\rightarrow \infty}\frac{\#\{(n,m):\,0\le n,m<q^J,\, m\in\mathcal{A},\, P(n)\equiv m\Mod{q^J}\}}{(q-1)^J}.\]
%We have that $\mathfrak{S}>0$ if and only if there is no integer $j$ such that $P(n)$ always has one of its final $j$ digits being equal to $a_0$.
\end{thrm}
Given a polynomial $P$ it is a straightforward computation to determine whether $\mathfrak{S}(P)>0$, in which case it takes infinitely many values in $\mathcal{A}$, or whether $\mathfrak{S}=0$ in which case it takes finitely many values in $\mathcal{A}$. (This is because $P(\mathbb{Z}_p)$ is a disjoint union of open balls and a finite set of points in the $p$-adic topology.) In particular, by Hensel lifting we see that Theorem \ref{thrm:Poly} shows that there are infinitely many $\ell^{th}$ powers in $\mathcal{A}$, provided that $q>\exp(\exp(2\ell))$.

Again we have made no particular effort to optimize the lower bound on $q$. It is clear that the statement must require $q$ to grow with $r$, since the main term $q^{k/r}(q-1)^k/q^k$ is only larger than $1$ if $q$ is large enough in terms of $r$. Presumably this bound would be improved if one used stronger bounds of Vinogradov type for the Weyl sums which appear rather than bounds based on Weyl-differencing for large $r$, and by less crude numerical bounds. We note that although the implied constant in the error term in the statement of the Theorem depends on the coefficients of $P$, the lower bound on $q$ depends only on the degree.

\begin{thrm}\label{thrm:ManyDigits}
Let $\epsilon>0$, $0<s<q^{1/5-\epsilon}$ and let $q$ be sufficiently large in terms of $\epsilon>0$. Let $b_1,\dots,b_s\in\{0,\dots,q-1\}$ be distinct and let $\mathcal{B}=\{\sum_{i=0}^{k-1}n_i q^i:n_i\in\{0,\dots,q-1\}\backslash\{b_0,b_1,\dots,b_s\}\}$ be the set of $k$-digit numbers in base $q$ with no digit in the set $\{b_1,\dots,b_s\}$. Then we have
\[\sum_{n<q^k}\Lambda(n)\mathbf{1}_{\mathcal{B}}(n)=\frac{q(\phi(q)-s')}{(q-1)\phi(q)}(q-s)^k+O_A\Bigl(\frac{(q-s)^k}{(\log{q^k})^A}\Bigr),\]
where $s'=\#\{1\le i\le s:(b_i,q)=1\}$.

Moreover, if $b_1,\dots,b_s$ are consecutive integers then the same result holds provided only that $q-s\ge q^{4/5+\epsilon}$ and $q$ is sufficiently large in terms of $\epsilon$.
\end{thrm}
In the case of $b_1,\dots,b_s$ consecutive with $q-s=q^{4/5+\epsilon}$ we see that Theorem \ref{thrm:ManyDigits} shows the existence of primes in a set containing $x^{4/5+\epsilon}$ elements less than $x$. The exponent $4/5$ is ultimately related to the $4/5$ exponent of Lemma \ref{lmm:PrimeSum} for an exponential sum over primes, and represents a limit of our basic method. As with Theorem \ref{thrm:Prime}, one would hope that utilizing Type I-II sums and Harman's sieve would extend this to sets of smaller density.

The conclusion of Theorem \ref{thrm:ManyDigits} holds in the case $q=10^8$ and $s=10$, so one can choose $\{b_1,\dots,b_{10}\}=\{0,11111111,22222222,\dots,99999999\}$. Thus there are infinitely many prime numbers with no string of 15 consecutive base 10 digits being the same. (Again, we expect 15 to be able to be reduced with slightly more effort.)

An analogous statement for the set $\mathcal{B}$ for polynomial values also holds, but in the more restrictive region $0<s<q^{1/r2^r-\epsilon}$ for arbitrary $b_1,\dots,b_s$ or $q-s\ge q^{1/r2^{r}+\epsilon}$ for consecutive $b_1,\dots,b_s$.
\section{Notation}
We use $e(x)=e^{2\pi i x}$ as the complex exponential and $\|x\|=\inf_{n\in\mathbb{Z}}|x-n|$ to denote the distance to the largest integer. We will use various expressions of the form $\min(A,\|\alpha\|^{-1})$, which are interpreted to take the value $A$ if $\|\alpha\|=0$. We use $n\sim N$ to abbreviate $n\in [N,2N)$. Any implied constants in asymptotic notation $\ll$ or $O(\cdot)$ are allowed to depend on the base $q$ and when dealing with polynomials as in Theorem \ref{thrm:Poly}, the polynomial $P$, but on no other quantity unless explicitly indicated by a subscript. Outside of Section \ref{sec:ExponentialSums} all quantities should be thought of as $k\rightarrow \infty$. In particular, $k$ will implicitly be assumed to be larger than any fixed constant.
\section{Outline}
We give an informal sketch the overall outline of the proof, which is essentially an application of the Hardy-Littlewood circle method.
We let $\hat{F}_{X}$ be the Fourier transform (over $\mathbb{Z}$) of the set $\mathcal{A}$ restricted to $\{1,\dots,X\}$. Thus for $X=q^k$ we have
\[\hat{F}_{q^k}(\theta)=\sum_{n\le q^k}\mathbf{1}_{A}(n)e(n\theta)=\prod_{i=0}^{k-1}\Bigl(\sum_{0\le n_i\le q-1}\mathbf{1}_{\mathcal{A}}(n_i)e(n_iq^i\theta)\Bigr).\]
Here we have written $n=\sum_{i=0}^{k-1}n_iq^i$. It is this factorization of $\hat{F}_{q^k}$ and the fact that the sum over $n_i$ is almost a geometric series which allows us very good Fourier control over $\mathcal{A}$. By Fourier inversion on $\mathbb{Z}/q^k\mathbb{Z}$
\[\mathbf{1}_{\mathcal{A}}(n)=\frac{1}{q^k}\sum_{0\le a<q^k}\hat{F}_{q^k}\Bigl(\frac{a}{q^k}\Bigr)e\Bigl(\frac{-a n}{q^k}\Bigr).\]
Thus
\[\sum_{n\le q^k}\Lambda(n)\mathbf{1}_{\mathcal{A}}(n)=\frac{1}{q^k}\sum_{0\le a<q^k}\hat{F}_{q^k}\Bigl(\frac{a}{q^k}\Bigr)S_{\Lambda,q^k}\Bigl(\frac{-a}{q^k}\Bigr),\]
where
\[S_{\Lambda,q^k}(\theta)=\sum_{n\le q^k}\Lambda(n)e(n\theta).\]
We split the contribution up depending on whether $a/q^k$ is close to a rational with small denominator or not. This distinguishes between those $a$ when $S_{\Lambda,q^k}(a/q^k)$ is large or not. It turns out that $\hat{F}_{q^k}(a/q^k)$ is large if $a$ is `close' to a number with few non-zero base $q$-digits, but these are somewhat rare and `spread out' except when $a/q^k$ close to a rational with denominator being a small power of $q$, and so it turns out decomposition is adequate for describing $\hat{F}_{q^k}$ as well as $S_{\Lambda,q^k}$.

If $\mathcal{D}$ is the set of $a$ such that $a/q^k=\ell/d+\beta$ for some integers $(\ell,d)=1$ of and some $\beta\in\mathbb{R}$ with $d|\beta|$ of size $D$, we use a $L^\infty$-$L^1$ bound to show their contribution is at most
\[\sup_{a\in\mathcal{D}}\Bigl|S_{\Lambda,q^k}\Bigl(\frac{a}{q^k}\Bigr)\Bigr| \sum_{a\in\mathcal{D}}\frac{1}{q^k}\Bigr|\hat{F}_{q^k}\Bigl(\frac{a}{q^k}\Bigr)\Bigr|.\]
One can save a small power of $D$ over the trivial bound on $S_{\Lambda,q^k}(a/q^k)$ for $a\in\mathcal{D}$. By using a large-sieve type argument (and the analytic description of $\hat{F}$) we show equidistribution for a truncated version of $\hat{F}_J$ of $\hat{F}_{q^k}$
\[\sum_{a/q^k=\ell/d+\beta}\Bigr|\hat{F}_{J}\Bigl(\frac{a}{q^k}\Bigr)\Bigr|\approx J\int_0^1|\hat{F}_{J}(\theta)|d\theta,\]
where $J=\#\mathcal{D}$. We then use the explicit analytic description of $\hat{F}_{q^k}$ to obtain a final bound which is unusually strong. In particular, we important make use of the averaging over different $\beta$. This bound loses only a small power of $D$ over the size of the largest individual terms in the sum. Crucially this power decreases to 0 as $q\rightarrow\infty$, whilst the power saving in $S_{\Lambda,q^k}$ was independent of $q$, and so we have an overall saving of a small power of $D$ if $q$ is sufficiently large. This saving shows that these `minor arc' contributions when $D$ is large are negligible.

Thus only those $a/q^k$ which are very close to a rational (i.e. $d|\beta|$ is small) make a noticeable contribution. In this case the problem simply reduces to estimating primes and elements of $\mathcal{A}$ separately in short intervals and arithmetic progressions. For primes this is well known, whilst for the set $\mathcal{A}$ this follows from a suitable $L^\infty$ bound on $\hat{F}$.

After writing this paper, the author discovered that very similar ideas appeared earlier in the literature, notably in \cite{MauduitRivat,Level1,Level2,Bourgain}. For simplicity we give an essentially self-contained proof, but emphasize to the reader that many Lemmas appearing are not new. It appears possible that (at least in the case when the base $q$ is large) that an argument similar to the one here might simplify or extend other arguments in the study of digit related functions. 

Much of the previous work relied on estimating correlations of primes with digit-related functions relied on exploiting a certain property of the Fourier transform described in \cite{MauduitGeneral} as the `carry property', which often allowed one to simplify bilinear expressions so the Fourier transform only relied on the lower-order digits. This feature is not present in our work.
\section{Exponential sums for primes and polynomials}\label{sec:ExponentialSums}
We first collect some results for exponential sums for primes and polynomials. The bounds here are well-known, but we give a essentially complete proofs since they differ slightly from some standard references.
\begin{lmm}\label{lmm:Equidistribution}
%Let $q^\delta<N,M$. If $a$ is not of the form $a/q^k=\ell/d+\beta$ for some $\ell,d\in\mathbb{Z}$ and $\beta\in\mathbb{R}$ with $(\ell,d)=1$ and $d<q^\delta$ and $d|\beta|<q^{\delta-k}$ then
%\[\sup_{\theta}\sum_{m\sim M}\min\Bigl(N, \Bigl\|\frac{a m}{q^k}+\theta\Bigr\|^{-1}\Bigr)\ll \frac{NM\log{M}}{q^\delta}.\]
Let $\alpha=a/d+\beta$ with $a,d$ coprime integers and $\beta\in\mathbb{R}$ satisfying $|\beta|<1/d^2$. Then we have
\[\sum_{n=1}^N\min\Bigl(M,\|\alpha n\|^{-1}\Bigr)\ll \Bigl(N+N M d|\beta|+\frac{1}{d|\beta|}+d\Bigr)\log{N}.\]
\end{lmm}
\begin{proof}
%By Dirichlet's approximation theorem, for any choice of $D$ there is an approximation $a/q^k=\ell/d+\beta$ for some integers $(\ell,d)=1$ with $d<D$ and $\beta\in\mathbb{R}$ with $|\beta|<1/d D$ (we include the possibility than $l=0$ and $d=1$). 
%
If $Nd|\beta|<1/2$ then we let $n=n_0+d n_1$ for non-negative integers $n_0,n_1$ with $n_0<d$ and $n_1<N/d$. If $n_0\ne 0$ then 
\[\|\alpha n\|=\|n_0a/d+ \beta n\|\ge \|n_0a/d\|-\|\beta n\|\ge \|n_0a/d\|/2\]
since $N|\beta|<1/2$. We let $b\in\{0,\dots,d-1\}$ be such that $b\equiv m_0a\Mod{d}$. Thus the terms with $n_0\ne 0$ contribute a total
\[\ll \sum_{n_1<N/d}\sum_{1\le b<\min(d,N)}\frac{d}{b}\ll \sum_{n_1<N/d}d\log{N}\ll (N+d)\log{N}.\]
The terms with $n_0=0$ contribute
\[\ll \sum_{1\le n_1<N/d}\min\Bigl(M,\|d n_1\beta\|^{-1}\Bigr)\ll \sum_{1\le n_1<N/d}\frac{1}{d n_1 |\beta|}\ll \frac{\log{N}}{d|\beta|}.\]
Here we have used the fact that since $n_0=0$ and we sum over $n\ge 1$ we must have $n_1\ge 1$.

We now consider the case $Nd|\beta|>1/2$. We let $n=n_0+d n_1+d\lfloor(d^2\beta)^{-1}\rfloor n_2$, with $0\le n_0<d$, $0\le n_1\le (d^2\beta)^{-1}$ and $0\le n_2\ll N d \beta$. Thus we obtain
\[\sum_{n=1}^N\min\Bigl(M, \|\alpha n\|^{-1}\Bigr)\ll\sum_{\substack{n_1\le 1/d^2\beta\\ n_2\ll N\beta/d}}\sum_{0\le n_0<d}\min\Bigl(N,\Bigl\|\theta+n_1d\beta+n_0(\ell/d+\beta)\Bigr\|^{-1}\Bigr)\]
where we have put $\theta=\beta d \lfloor(d^2\beta)^{-1}\rfloor m_2$ for convenience. The inner sum is of the form $\sum_i\min(N,\|\theta_i\|^{-1})$ for $d$ points $\theta_i$ which are $1/2d$ separated. Therefore the sum over $m_0$ is
\begin{align*}
&\ll d\log{d}+\sup_{0\le m_0<d}\min\Bigl(N,\|\theta+m_1d\beta+m_0(\ell/d+\beta)\|^{-1}\Bigr)\\
&\ll d\log{N}+\sup_{0\le \epsilon<1}\min\Bigl(N,\frac{d}{\|d\theta+(m_1+O(1))d^2\beta\|}\Bigr)
\end{align*}
since $\|t\|^{-1}\le d\|dt\|^{-1}$ for all $t$. The term $d\log{d}$ contributes $\ll (M+d)\log{N}$ to the total sum, which is acceptable. Thus we are left to bound
\[\sum_{m_2\ll Md\beta}\sup_{\substack{\theta\in\mathbb{R}}}\sum_{m_1\le 1/d\beta}\min\Bigl(N,\frac{d}{\|\theta+(m_1+O(1))d^2\beta\|}\Bigr).\]
 The inner sum is of the form ($O(1)$ copies of) $\sum_i\min(M,\|\theta_i\|^{-1})$ for $O(1/d^2\beta)$ points $\theta_i$ which are $d^2\beta$-separated $\mod{1}$. Therefore the inner sum is $\ll(M+d/d^2\beta)\log{N}$, and this gives a bound
\[\ll Nd|\beta|\Bigl(M+\frac{1}{d|\beta|}\Bigr)\log{N}\ll \Bigl(MN d|\beta|+N\Bigr)\log{N}.\]
Putting these bounds together gives
\[\sum_{n=1}^N\min\Bigl(M,\|\alpha n\|^{-1}\Bigr)\ll \Bigl(N+NM d|\beta|+\frac{1}{d|\beta|}+d\Bigr)\log{N}.\qedhere\]
%Choosing $D=q^\delta$ then gives the bound of the Lemma unless $|\beta|<q^{k-\delta}$, in which case $a/q^k$ is of the form $\ell/d+\beta$ with $d|\beta|<q^{\delta-k}$ and $d<q^\delta$, as required.
\end{proof}
\begin{lmm}\label{lmm:PrimeSum}
Let $\alpha=a/d+\beta$ with $(a,d)=1$ and $|\beta|<1/d^2$. Then
\[S_{\Lambda,x}(\alpha)=\sum_{n<x}\Lambda(n)e(n\alpha)\ll \Bigl(x^{4/5}+\frac{x^{1/2}}{|d\beta|^{1/2}}+x|d\beta|^{1/2}\Bigr)(\log{x})^4.\]
\end{lmm}
\begin{proof}
From \cite[(6), Page 142]{Davenport}, taking $f(n)=e(n\alpha)$ we have that for any choice of $U,V\ge 2$ with $UV\le x$
\begin{align*}
\sum_{n<x}\Lambda(n)e(n\alpha)&\ll U+(\log{x})\sum_{1\le t<UV}\sup_w\Bigl|\sum_{w<r\le x/t}e(rt\alpha)\Bigr|\\
&+x^{1/2}(\log{x})^3\sup_{\substack{U\le M\le x/V\\ V\le j\le N/M}}\Bigl(\sum_{V<k<x/M}\Bigl|\sum_{\substack{M<m\le 2M\\ m\le x/k\\ m\le x/j}}e(\alpha m (j-k))\Bigr|\Bigr)^{1/2}.
\end{align*}
The sum over $r$ is clearly $\ll \min(x/t,\|t\alpha\|^{-1})$ and the sum over $m$ is similarly $\ll \min(M,\|(j-k)\alpha\|^{-1})$. Putting $t$ and $j-k$ into dyadic intervals and applying Lemma \ref{lmm:Equidistribution} to the resulting sums (or the trival bound when $j=k$) gives a bound
\[\ll \Bigl(UV+xd|\beta|+\frac{1}{d|\beta|}+d+\frac{x}{U^{1/2}}+\frac{x}{V^{1/2}}+x|d\beta|^{1/2}+\frac{x^{1/2}}{|d\beta|^{1/2}}+x^{1/2}d^{1/2}\Bigr)(\log{x})^4.\]
Choosing $U=V=x^{2/5}$ and simplifying the terms then gives the result.
\end{proof}

\begin{lmm}\label{lmm:PolySum}
Let $P\in\mathbb{Z}[X]$ be an integer polynomial of degree $r\ge 2$ with lead coefficient $a_r$. Let $\alpha\in\mathbb{R}$ be such that $a_rr!\alpha=a/d+\beta$ with $(a,d)=1$ and $|\beta|<1/d^2$. Then for any constant $\epsilon>0$ we have
\[S_{P,x}(\alpha)=\sum_{P(n)<x}e(\alpha P(n))\ll_{\epsilon} x(\log{x})\Bigl(\frac{1}{x}+\frac{1}{x^r d|\beta|}+d|\beta|+\frac{d}{x^{r}}\Bigr)^{1/2^{r}}.\]
\end{lmm}
\begin{proof}
If $\mathcal{I}$ is an interval contained in $[0,x]$ then 
\begin{align*}
\Bigl|\sum_{n\in\mathcal{I}}e(\alpha P(n))\Bigr|^2=\sum_{|h|<x}\sum_{\substack{n\in\mathcal{I}\\ n+h\in\mathcal{I}}}e(\alpha(P(n+h)-P(n)))&\le \sum_{|h|<x}\Bigl|\sum_{n\in\mathcal{I}(h)}e(\alpha Q_h(n))\Bigr|
\end{align*}
where $\mathcal{I}(h)=\mathcal{I}\cap(\mathcal{I}-h)$ is an interval contained in $[0,x]$, and $Q_h(n)=P(n+h)-P(n)$ is a polynomial of degree $r-1$ with lead coefficient $a_r r h$. Applying this and Cauchy's inequality $r-1$ times gives
\begin{align*}
\Bigl|\sum_{n\in\mathcal{I}}e(\alpha P(n))\Bigr|^{2^{r-1}}&\le (2x)^{2^{r-1}-r}\sum_{|h_1|,\dots,|h_{r-1}|<x}\Bigl|\sum_{n\in\mathcal{I}(h_1,\dots,h_{r-1})}e(\alpha r! h_1\dots h_{r-1}n )\Bigr|\\
&\ll x^{2^{r-1}-r}\sum_{H<x^{r-1}}\tau_{r-1}(H)\min(x,\|a_r r!H\|^{-1})
\end{align*}
where we have put $H=h_1\dots h_r$. We split the sum depending on whether $\tau_{r-1}(H)>B$ or not, for some quantity $B$ which we choose later. This shows that the inner sum is of size
\begin{align*}
&\ll \sum_{\substack{H<x^{r-1}\\ \tau_{r-1}(H)<B}}B\min(x,\|\alpha a_rr!H\|^{-1})+\sum_{\substack{H<x^{r-1}\\ \tau_{r-1}(H)>B}}x\frac{\tau_{r-1}(H)^2}{B}\\
&\ll B\Bigl(x^{r-1}+x^rd|\beta|+\frac{1}{d|\beta|}+d\Bigr)\log{x}+\frac{x^r(\log{x})^{(r-1)^2}}{B}
\end{align*}
by applying Lemma \ref{lmm:Equidistribution}. Writing this bound as $x^r B/ Z+x^r(\log{x})^{(r-1)^2}/B$ and choosing $B=Z^{1/2}$ then gives the result, noting that $(\log{x})^{(r-1)^2/2^{r-1}}<\log{x}$.
\end{proof}
\section{Fourier analysis}\label{sec:Fourier}
We now establish in turn several properties of the function $\hat{F}_{q^k}$, which are the key ingredient in our result.
\begin{lmm}[$L^1$ bound]\label{lmm:L1Bound}
There exists a constant $C_q\in [1/\log{q},1+3/\log{q}]$ such that
\[\sup_{\theta\in\mathbb{R}}\sum_{0\le a< q^k}\Bigl|\hat{F}_{q^k}\Bigl(\theta+\frac{a}{q^k}\Bigr)\Bigr|\ll (C_q q\log{q})^k.\]
\end{lmm}
\begin{proof}
We expand out the definition of $\hat{F}_{q^k}$, and let $n=\sum_{i=0}^{k-1}n_i q^i$ be the base-$q$ expansion of $n$.
\begin{align*}
\hat{F}_{q^k}(t)&=\sum_{n<q^k}\mathbf{1}_{\mathcal{A}}(n)e(t n)=\prod_{i=0}^{k-1}\Bigl(\sum_{n_i=0}^{q-1}\mathbf{1}_{\mathcal{A}}(n_i)e(n_i q^i t)\Bigr).
\end{align*}
The sum over $n_i$ is a sum over all values in $\{0,\dots,q-1\}\backslash\{a_0\}$, and so is bounded by
\begin{equation}
\Bigl|\frac{e(q^{i+1}t)-1}{e(q^it)-1}-e(a_0 q^i t)\Bigr|\le \min\Bigl(q,1+\frac{1}{2\|q^i t\|}\Bigr).\label{eq:LittleSum}
\end{equation}
For $t\in[0,1)$, we write $t=\sum_{i=1}^{k} t_i q^{-i}+\epsilon$ with $t_1,\dots,t_k\in\{0,\dots,q-1\}$ and $\epsilon\in[0,1/q^k)$. We see that $\|q^i t\|^{-1}=\|t_{i+1}/q+\epsilon_i\|^{-1}$ for some $\epsilon_i\in[0,1/q)$. In particular, $\|q^i t\|^{-1}\le\max(q/t_{i+1},q/(q-1-t_{i+1}))$ if $t_{i+1}\ne 0,q-1$. Thus we see that
\begin{align*}
\sup_{\theta\in\mathbb{R}}\sum_{0\le a <q^k}\Bigl|\hat{F}_{q^k}\Bigl(\theta+\frac{a}{q^k}\Bigr)\Bigr|&\ll \sum_{t_1,\dots,t_k<q}\prod_{i=1}^k\min\Bigl(q,1+\max\Bigl(\frac{q}{2t_i},\frac{q}{2(q-1-t_i)}\Bigr)\Bigr)\\
&\ll \prod_{i=1}^k \Bigl(3q+\sum_{1\le t_i\le (q-1)/2}\frac{q}{t_i}\Bigr)\\
&\ll (3q+q\log{q})^k.
\end{align*}
Here we used a small computation to verify $\sum_{1\le t\le (q-1)/2}t^{-1}\le \log{q}$ for all integers $q<20$, whilst for $q\ge 20>2/(\log{2}-\gamma)$ (where $\gamma$ is Euler's constant), we have $\sum_{1\le t\le (q-1)/2}t^{-1}\le \log{q}-\log{2}+\gamma+2/q\le \log{q}$. (This bound is only relevant to the final lower bound on $q$; for a qualitative statement a bound $O(\log{q})$ suffices.)
\end{proof}
\begin{lmm}[Large sieve estimate]\label{lmm:TypeI}
We have
\[\sup_{\theta\in\mathbb{R}}\sum_{d\sim D}\sum_{\substack{0<\ell<d\\ (\ell,d)=1}}\sup_{|\epsilon|<\frac{1}{10D^2}}\Bigl|\hat{F}_{q^k}\Bigl(\frac{\ell}{d}+\theta+\epsilon\Bigr)\Bigr|\ll (D^2+q^k)(C_q\log{q})^k.\]
Here $C_q$ is the constant described in Lemma \ref{lmm:L1Bound}.
\end{lmm}
\begin{proof}
We have that
\[\hat{F}_{q^k}(t)=\hat{F}_{q^k}(u)+\int_t^u\hat{F}_{q^k}'(v)dv.\]
Thus integrating over $u\in [t-\delta,t+\delta]$ we have
\[|\hat{F}_{q^k}(t)|\ll \frac{1}{\delta}\int_{t-\delta}^{t+\delta}|\hat{F}_{q^k}(u)|du+\int_{t-\delta}^{t+\delta}|\hat{F}_{q^k}'(u)|du.\]
We note that the fractions $\ell/d+\theta+\epsilon$ with $(\ell,d)=1$, $d<2D$ and $|\epsilon|<1/10D^2$ are separated from one another by $\gg 1/D^2$ . Thus
\[\sum_{d\sim D}\sum_{\substack{0<\ell<d\\ (\ell,d)=1}}\sup_{|\epsilon|<1/10D^2}\Bigl|\hat{F}_{q^k}\Bigl(\frac{\ell}{d}+\theta+\epsilon\Bigr)\Bigr|\ll D^2\int_0^1|\hat{F}_{q^k}(u)|du+\int_{0}^{1}|\hat{F}_{q^k}'(u)|du.\]
We note that, writing $n=\sum_{i=0}^{k-1}n_i q^i$ we have
\begin{align*}
\hat{F}_{q^k}'(t)&=2\pi i\sum_{n\le q^k}n\mathbf{1}_{\mathcal{A}}(n)e(n t)\\
&=2\pi i \sum_{j=0}^{k-1}q^j\Bigl(\sum_{0\le n_j<q}n_j\mathbf{1}_{\mathcal{A}}e(n_jq^{j}t)\Bigr)\prod_{i\ne j}\Bigl(\sum_{0\le n_i<q}\mathbf{1}_{\mathcal{A}}(n_i)e(n_iq^it)\Bigr).
\end{align*}
Thus, as in Lemma \ref{lmm:L1Bound}, we have
\[|\hat{F}_{q^k}'(t)|\ll \sum_{j=0}^{k-1}q^{j+1}\prod_{i\ne j}\min\Bigl(q,1+\frac{1}{2\|q^it\|}\Bigr)\ll q^k\prod_{i=0}^{k-1}\min\Bigl(q,1+\frac{1}{2\|q^it\|}\Bigr),\]
and we have the same bound for $|\hat{F}_{q^k}(t)|$ but without the $q^k$ factor. We let $t=\sum_{i=1}^kt_iq^{-k}+\epsilon$ for some $t_1,\dots,t_k\in\{0,\dots,q-1\}$ and $\epsilon\in[0,1/q^k)$. We see that, as in Lemma \ref{lmm:L1Bound} we have
\begin{align*}
\int_0^1\prod_{i=0}^{k-1}\min\Bigl(q,1+\frac{1}{2\|q^it\|}\Bigr)dt&\ll \frac{1}{q^k}\sum_{t_1,\dots,t_k<q}\prod_{i=0}^{k-1}\Bigl(1+\min\Bigl(q,\frac{q}{2t_i},\frac{q}{2(q-1-t_i)}\Bigr)\Bigr)\\
&\ll (C_q\log{q})^k.
\end{align*}
Putting this all together then gives the result.
\end{proof}
\begin{lmm}[Hybrid estimate]\label{lmm:ExtendedTypeI}
Let $B,D\gg 1$. Then we have
\[\sum_{d\sim D}\sum_{\substack{\ell<d\\ (\ell,d)=1}}\sum_{\substack{|\eta|<B\\ q^k\ell/d+\eta\in\mathbb{Z}}}\Bigl|\hat{F}_{q^k}\Bigl(\frac{\ell}{d}+\frac{\eta}{q^k}\Bigr)\Bigr|\ll (q-1)^k(D^2B)^{\alpha_q}+D^2B(C_q\log{q})^k,\]
where $C_q$ is the constant described in Lemma \ref{lmm:L1Bound} and
\[\alpha_q=\frac{\log\Bigl(C_q\frac{q}{q-1}\log{q}\Bigr)}{\log{q}}.\]
\end{lmm}
\begin{proof}
The result follows immediately from Lemma \ref{lmm:L1Bound} if $B>q^k$, so we may assume $B<q^k$. For any integer $k_1\in [0,k]$ we have
\begin{align*}
\hat{F}_{q^k}(\alpha)&=\prod_{i=0}^{k-k_1-1}\Bigl(\sum_{n_i<q}\mathbf{1}_{\mathcal{A}}(n_i)e(n_iq^i\alpha)\Bigr)\prod_{i=k-k_1}^{k-1}\Bigl(\sum_{n_i<q}\mathbf{1}_{\mathcal{A}}(n_i)e(n_iq^i\alpha)\Bigr)\\
&=\hat{F}_{q^{k-k_1}}(\alpha)\hat{F}_{q^{k_1}}(q^{k-k_1}\alpha).
\end{align*}
Using this and the trivial bound $|\hat{F}_{q^j}(\theta)|\le (q-1)^j$,  for $k_1+k_2\le k$ we have that
\[\Bigl|\hat{F}_{q^k}\Bigl(\frac{\ell}{d}+\frac{\eta}{q^k}\Bigr)\Bigr|\le (q-1)^{k-k_1-k_2}\Bigl|\hat{F}_{q^{k_1}}\Bigl(\frac{q^{k-k_1}\ell}{d}+\frac{\eta}{q^{k_1}}\Bigr)\Bigr|\sup_{|\epsilon|\le B/q^k}\Bigl|\hat{F}_{q^{k_2}}\Bigl(\frac{\ell}{d}+\epsilon\Bigr)\Bigr|.\]
Substituting this bound gives
\begin{align*}
\sum_{d\sim D}&\sum_{\substack{\ell<d\\ (\ell,d)=1}}\sum_{\substack{|\eta|<B\\ q^k\ell/d+\eta\in\mathbb{Z}}}\Bigl|\hat{F}_{q^k}\Bigl(\frac{\ell}{d}+\frac{\eta}{q^k}\Bigr)\Bigr|\ll (q-1)^{k-k_1-k_2}\\
&\times\sum_{d\sim D}\sum_{\substack{\ell<d\\ (\ell,d)=1}}\sup_{|\epsilon|<B/q^k}\Bigl|\hat{F}_{q^{k_2}}\Bigl(\frac{\ell}{d}+\epsilon\Bigr)\Bigr|\sum_{\substack{|\eta|<B\\ q^k\ell/d+\eta\in\mathbb{Z}}}\Bigl|\hat{F}_{q^{k_1}}\Bigl(\frac{q^{k-k_1}\ell}{d}+\frac{\eta}{q^{k_1}}\Bigr)\Bigr|.
\end{align*}
We choose $k_1$ minimally such that $q^{k_1}>B$, and extend the inner sum to $|\eta|<q^{k_1}$. Applying Lemma \ref{lmm:L1Bound} to the inner sum, and then Lemma \ref{lmm:TypeI} to the sum over $d,\ell$ gives
\[\sum_{d\sim D}\sum_{\substack{\ell<d\\ (\ell,d)=1}}\sum_{\substack{|\eta|<B\\ q^k\ell/d+\eta\in\mathbb{Z}}}\Bigl|\hat{F}_{q^k}\Bigl(\frac{\ell}{d}+\frac{\eta}{q^k}\Bigr)\Bigr|\ll (q-1)^{k-k_1-k_2}q^{k_1}(q^{k_2}+D^2)(C_q\log{q})^{k_1+k_2}.\]
We choose $k_2=\min(k-k_1,\lfloor2\log{D}/\log{q}\rfloor)$. We see that
\begin{align*}
\Bigl(\frac{C_q q\log{q}}{q-1}\Bigr)^{k_1+k_2}&\ll (D^2B)^{\alpha_q},\\
D^2 q^k_1\Bigl(\frac{C_q\log{q}}{q-1}\Bigr)^{k_1+k_2}&\ll \frac{D^2 B}{(q-1)^k}(C_q\log{q})^k+(D^2B)^{\alpha_q}.
\end{align*}
Combining these bounds gives the result.
\end{proof}
\begin{lmm}[$L^\infty$ bound]\label{lmm:LInfBound}
Let $d<q^{k/3}$ be of the form $d=d_1d_2$ with $(d_1,q)=1$ and $d_1\ne 1$, and let $|\epsilon|<1/2q^{2k/3}$. Then for any integer $\ell$ coprime with $d$ we have 
\[\Bigl| \hat{F}_{q^k}\Bigl(\frac{\ell}{d}+\epsilon\Bigr)\Bigr|\ll (q-1)^{k}\exp\Bigl(-c_q\frac{k}{\log{d}}\Bigr)\]
for some constant $c_q>0$ depending only on $q$.
\end{lmm}
\begin{proof}
We have that
\[|e(n\theta)+e((n+1)\theta)|^2=2+2\cos(2\pi \theta)<4\exp(-2\|\theta\|^2).\]
This implies that
\[\Bigl|\sum_{n_i<q}\mathbf{1}_{\mathcal{A}}(n_i)e(n_i \theta)\Bigr|\le q-3+2\exp(-\|\theta\|^2)\le (q-1)\exp\Bigl(-\frac{\|\theta\|^2}{q}\Bigr).\]
We substitute this bound into our expression for $\hat{F}$, which gives
\begin{align*}
\Bigl|\hat{F}_{q^k}\Bigl(\frac{\ell}{d}\Bigr)\Bigr|&=\prod_{i=0}^{k-1}\Bigl|\sum_{n_i<q}\mathbf{1}_{\mathcal{A}}(n_i)e(n_iq^it)\Bigr|\\
%&=\Bigl|\hat{F}_{q^J}\Bigl(\frac{\ell}{d}\Bigr)\Bigr|\prod_{i=J}^{k-1}\Bigl|\sum_{n_i<q}\mathbf{1}_{\mathcal{A}}(n_i)e(n_iq^it)\Bigr|\\
&\le %\Bigl|\hat{F}_{q^J}\Bigl(\frac{\ell}{d}\Bigr)\Bigr|
(q-1)^{k}\exp\Bigl(-\frac{1}{q}\sum_{i=0}^{k-1}\|q^it\|^2\Bigr).
\end{align*}
If $\|q^it\|<1/2q$ then $\|q^{i+1}t\|=q\|q^i t\|$. If $t=\ell/d_1d_2$ with $d_1>1$, $(d_1,q)=1$ and $(\ell,d_1)=1$ then $\|q^it\|\ge 1/d$ for all $i$. Similarly, if $t=\ell/d_1d_2+\epsilon$ with $\ell,d_1,d_2$ as above $|\epsilon|<q^{-2k/3}/2$ and $d=d_1d_2<q^{k/3}$ then for $i<k/3$ we have $\|q^it\|\ge 1/d-q^i|\epsilon|\ge 1/2d$. Thus, for any interval $\mathcal{I}\subseteq[0,k/3]$ of length $\log{d}/\log{q}$, there must be some integer $i\in\mathcal{I}$ such that $\|q^i(\ell/d+\epsilon)\|>1/2q^2$. This implies that
\[\sum_{i=0}^k\Bigl\|q^i\Bigl(\frac{\ell}{d}+\epsilon\Bigr)\Bigr\|^2\ge \frac{1}{4q^4}\Bigl\lfloor\frac{k\log{q}}{3\log{d}}\Bigr\rfloor.\]
Substituting this into the bound for $\hat{F}$, and recalling we assume $d<q^{k/3}$ gives the result.
\end{proof}
\section{Minor arcs}
We now use the exponential sum estimates from the previous sections to show that when $\alpha$ is `far' from a rational with small denominator the quantity $\hat{F}_{q^k}(\alpha)S_{\Lambda,q^k}(-\alpha)$ and $\hat{F}_{q^k}(\alpha)S_{P,q^k}(-\alpha)$ are typically small in absolute value.
\begin{lmm}\label{lmm:PrimeMinor}
Let $1\ll B\ll q^k/D_0D$ and $1\ll D\ll D_0\ll q^{k/2}$. Then we have
\begin{align*}
\sum_{d\sim D}\sum_{\substack{0<\ell<d\\ (\ell,d)=1}}\sum_{\substack{|\eta|\sim B\\ q^k\ell/d+\eta\in\mathbb{Z}}}\Bigl|\hat{F}_{q^k}\Bigl(\frac{\ell}{d}+\frac{\eta}{q^k}\Bigr)S_{\Lambda,q^k}\Bigl(-\frac{\ell}{d}-\frac{\eta}{q^k}\Bigr)\Bigr|\\
\ll k^4(q-1)^kq^{k}\Bigl(\frac{1}{(DB)^{1/5-\alpha_q}}+\frac{q^{k\alpha_q}}{D_0^{1/2}}\Bigr).
\end{align*}
and
\begin{align*}
\sum_{d\sim D}\sum_{\substack{0<\ell<d\\ (\ell,d)=1}}\sum_{\substack{|\eta|\ll 1\\ q^k\ell/d+\eta\in\mathbb{Z}}}\Bigl|\hat{F}_{q^k}\Bigl(\frac{\ell}{d}+\frac{\eta}{q^k}\Bigr)S_{\Lambda,q^k}\Bigl(-\frac{\ell}{d}-\frac{\eta}{q^k}\Bigr)\Bigr|\\
\ll k^4(q-1)^kq^{k}\Bigl(\frac{1}{D^{1/5-\alpha_q}}+\frac{D_0^{1/2+2\alpha_q}}{q^{k/2}}\Bigr).
\end{align*}
Here $\alpha_q$ is the constant described in Lemma \ref{lmm:ExtendedTypeI}.
\end{lmm}
\begin{proof}
By Lemma \ref{lmm:ExtendedTypeI} we have that if $D^2B\ll q^k$ then
\[\sum_{d\sim D}\sum_{\substack{\ell<d\\ (\ell,d)=1}}\sum_{\substack{|\eta|<B\\ q^k\ell/d+\eta\in\mathbb{Z}}}\Bigl|\hat{F}_{q^k}\Bigl(\frac{\ell}{d}+\frac{\eta}{q^k}\Bigr)\Bigr|\ll (q-1)^k(D^2B)^{\alpha_q}.\]
By Lemma \ref{lmm:PrimeSum} we have
\[\sup_{\substack{d\sim D\\ (\ell,d)=1\\ |\eta|\sim B}}\Bigl|\sum_{n<q^k}\Lambda(n)e\Bigl(-n\Bigl(\frac{\ell}{d}+\frac{\eta}{q^k}\Bigr)\Bigr)\Bigr|\ll \Bigl(q^{4k/5}+\frac{q^{k}}{(DB)^{1/2}}+\frac{(DB)^{1/2}}{q^{k/2}}\Bigr)(k\log{q})^4.\]
Putting these together gives
\begin{align*}
&\sum_{d\sim D}\sum_{\substack{0<\ell<d\\ (\ell,d)=1}}\sum_{\substack{|\eta|\sim B\\ q^k\ell/d+\eta\in\mathbb{Z}}}\Bigl|\hat{F}_{q^k}\Bigl(\frac{\ell}{d}+\frac{\eta}{q^k}\Bigr)S_{\Lambda,q^k}\Bigl(-\frac{\ell}{d}-\frac{\eta}{q^k}\Bigr)\Bigr|\\
&\ll  k^4 q^k(q-1)^k\Bigl(\frac{(D^2B)^{\alpha_q}}{q^{k/5}}+\frac{(D^2B)^{\alpha_q}}{(DB)^{1/2}}+\frac{(DB)^{1/2}(D^2B)^{\alpha_q}}{q^{k/2}}\Bigr).
\end{align*}
Recalling that $D^2B<q^k$ and $DB<q^k/D_0$ by assumption, we see that this is
\[\ll k^4 q^k(q-1)^k \Bigl((D^2B)^{\alpha_q-1/5}+(D^2B)^{\alpha_q-1/4}+\frac{q^{k\alpha_q}}{D_0^{1/2}}\Bigr),\]
and the first term clearly dominates the second. 

By partial summation we see that we obtain the same bound for $S_{\Lambda,q^k}(\alpha+O(1/q^k))$ as the bound for $S_{\Lambda,q^k}(\alpha)$ given in Lemma \ref{lmm:PrimeSum}. Thus in the case $|\eta|\ll 1$ we obtain the well-known bound 
\[\sup_{\substack{d\sim D\\ (\ell,d)=1\\ |\eta|\ll 1}}\Bigl|\sum_{n<q^k}\Lambda(n)e\Bigl(-n\Bigl(\frac{\ell}{d}+\frac{\eta}{q^k}\Bigr)\Bigr)\Bigr|\ll \Bigl(q^{4k/5}+\frac{q^{k}}{D^{1/2}}+\frac{D^{1/2}}{q^{k/2}}\Bigr)(k\log{q})^4.\]
This gives
\begin{align*}
&\sum_{d\sim D}\sum_{\substack{0<\ell<d\\ (\ell,d)=1}}\sum_{\substack{|\eta|\ll 1\\ q^k\ell/d+\eta\in\mathbb{Z}}}\Bigl|\hat{F}_{q^k}\Bigl(\frac{\ell}{d}+\frac{\eta}{q^k}\Bigr)S_{\Lambda,q^k}\Bigl(-\frac{\ell}{d}-\frac{\eta}{q^k}\Bigr)\Bigr|\\
&\ll  k^4 q^k(q-1)^k\Bigl(\frac{D^{2\alpha_q}}{q^{k/5}}+\frac{D^{2\alpha_q}}{D^{1/2}}+\frac{D^{1/2+2\alpha_q}}{q^{k/2}}\Bigr).
\end{align*}
Recalling that $1\ll D\ll D_0\ll q^{k/2}$ then gives the result.
\end{proof}
\begin{lmm}\label{lmm:PolyMinor}
Let $DB\ll q^k/D_0$ and $D\ll D_0\ll q^{k/2}$. Then we have
\begin{align*}
\sum_{d\sim D}\sum_{\substack{0<\ell<da_r r!\\ (\ell,d)=1}}\sum_{\substack{|\eta|\sim B\\ q^k\ell/d+\eta\in\mathbb{Z}}}\Bigl|\hat{F}_{q^k}\Bigl(\frac{\ell}{d a_r r!}+\frac{\eta}{q^ka_r r!}\Bigr)S_{P,q^k}\Bigl(\frac{-\ell}{da_r r!}+\frac{-\eta}{q^k a_r r!}\Bigr)\Bigr|\\
\ll k (q-1)^k q^{k/r}\Bigl(\frac{1}{(DB)^{1/r2^r-\alpha_q}}+\frac{q^{k\alpha_q}}{D_0^{1/2^{r}}}\Bigr),
\end{align*}
and
\begin{align*}
\sum_{d\sim D}\sum_{\substack{0<\ell<da_r r!\\ (\ell,d)=1}}\sum_{\substack{|\eta|\ll 1\\ q^k\ell/d+\eta\in\mathbb{Z}}}\Bigl|\hat{F}_{q^k}\Bigl(\frac{\ell}{d a_r r!}+\frac{\eta}{q^ka_r r!}\Bigr)S_{P,q^k}\Bigl(\frac{-\ell}{da_r r!}+\frac{-\eta}{q^k a_r r!}\Bigr)\Bigr|\\
\ll k (q-1)^k q^{k/r}\Bigl(\frac{1}{D^{1/r2^r-\alpha_q}}+\frac{D_0^{2\alpha_q+1/2^r}}{q^{k/2^r}}\Bigr).
\end{align*}
Here $\alpha_q$ is the constant described in Lemma \ref{lmm:ExtendedTypeI}.
\end{lmm}
\begin{proof}
By Lemma \ref{lmm:ExtendedTypeI} we have that if $D^2B\ll q^k$ then
\[\sum_{d\sim a_r r! D}\sum_{\substack{\ell<d\\ (\ell,d)=1}}\sum_{\substack{|\eta|<B\\ q^k\ell/d+\eta\in\mathbb{Z}}}\Bigl|\hat{F}_{q^k}\Bigl(\frac{\ell}{d}+\frac{\eta}{q^k}\Bigr)\Bigr|\ll (q-1)^k(D^2B)^{\alpha_q}.\]
By Lemma \ref{lmm:PolySum} we have
\[\sup_{\substack{d\sim D\\ (\ell,d)=1\\ |\eta|\sim B}}\Bigl|\sum_{P(n)<q^{k}}e\Bigl(\frac{-P(n)}{a_rr!}\Bigl(\frac{\ell}{d}+\frac{\eta}{q^k}\Bigr)\Bigr)\Bigr|\ll k q^{k/r}\Bigl(\frac{1}{q^{k/r2^{r}}}+\frac{1}{(DB)^{1/2^{r}}}+\frac{(DB)^{1/2^{r}}}{q^{k/2^{r}}}\Bigr).\]
Putting these together gives
\begin{align}
\sum_{d\sim D}\sum_{\substack{0<\ell<da_r r!\\ (\ell,d)=1}}&\sum_{\substack{|\eta|\sim B\\ q^k\ell/d+\eta\in\mathbb{Z}}}\Bigl|\hat{F}_{q^k}\Bigl(\frac{\ell}{d a_r r!}+\frac{\eta}{q^ka_r r!}\Bigr)\sum_{P(n)<q^{k}}e\Bigl(P(n)\Bigl(\frac{-\ell}{da_r r!}+\frac{-\eta}{q^k a_r r!}\Bigr)\Bigr)\Bigr|\nonumber\\
&\ll k q^{k/r}(q-1)^k\Bigl(\frac{(D^2B)^{\alpha_q}}{q^{k/r2^{r}}}+\frac{(D^2B)^{\alpha_q}}{(DB)^{1/2^{r}}}+\frac{(DB)^{1/2^{r}}(D^2B)^{\alpha_q}}{q^{k/2^{r}}}\Bigr).\label{eq:PolyBound}
\end{align}
Recalling that $D^2B<q^k$ and $DB<q^k/D_0$ by assumption, we see that this is
\[\ll k q^{k/r}(q-1)^k \Bigl((D^2B)^{\alpha_q-1/r2^{r}}+(D^2B)^{\alpha_q-1/2^{r+1}}+\frac{q^{k\alpha_q}}{D_0^{1/2^{r}}}\Bigr),\]
and the first term clearly dominates the second. As in Lemma \ref{lmm:PrimeMinor}, in the case we instead sum over $|\eta|\ll 1$, we obtain the same bound as \eqref{eq:PolyBound} with $B$ replaced by 1, since by partial summation we obtain the bound of Lemma \ref{lmm:PolySum} for $S_{P,q^k}(\alpha)$ as $S_{P,q^k}(\alpha+O(1/q^k))$. This gives
\begin{align*}
\sum_{d\sim D}\sum_{\substack{0<\ell<da_r r!\\ (\ell,d)=1}}&\sum_{\substack{|\eta|\ll 1\\ q^k\ell/d+\eta\in\mathbb{Z}}}\Bigl|\hat{F}_{q^k}\Bigl(\frac{\ell}{d a_r r!}+\frac{\eta}{q^ka_r r!}\Bigr)\sum_{P(n)<q^{k}}e\Bigl(P(n)\Bigl(\frac{-\ell}{da_r r!}+\frac{-\eta}{q^k a_r r!}\Bigr)\Bigr)\Bigr|\nonumber\\
&\ll k q^{k/r}(q-1)^k\Bigl(\frac{D^{2\alpha_q}}{q^{k/r2^{r}}}+\frac{D^{2\alpha_q}}{D^{1/2^{r}}}+\frac{D^{1/2^{r}+2\alpha_q}}{q^{k/2^{r}}}\Bigr).
\end{align*}
Recalling $1\ll D\ll D_0\ll q^{k/2}$ gives the result.
\end{proof}
\section{Major Arcs}
We now consider $\hat{F}_{q^k}(\alpha)S_{\Lambda,q^k}(-\alpha)$ and $\hat{F}_{q^k}(\alpha)S_{P,q^k}(-\alpha)$ when $\alpha$ is close to a rational with small denominator.
\begin{lmm}\label{lmm:MajorError}
Let $D$, $B\ll \exp(c_q^{1/2} k^{1/2}/3)$ where $c_q$ is the constant from Lemma \ref{lmm:LInfBound}. Then we have
\[\sum_{\substack{d<D\\ \exists p|d,p\nmid q}}\sum_{\substack{0<\ell<d\\ (\ell,d)=1}}\sum_{\substack{|\eta|\ll B\\ q^k\ell/d+\eta\in\mathbb{Z}}}\Bigl|\hat{F}_{q^k}\Bigl(\frac{\ell}{d}+\frac{\eta}{q^k}\Bigr)S_{\Lambda,q^k}\Bigl(\frac{-\ell}{d}+\frac{-\eta}{q^k}\Bigr)\Bigr|\ll \frac{q^k(q-1)^k}{\exp(c_q^{1/2}k^{1/2})},\]
and
\[\sum_{\substack{d<D\\ \exists p|d,p\nmid q}}\sum_{\substack{0<\ell<d\\ (\ell,d)=1}}\sum_{\substack{|\eta|\ll B\\ q^k\ell/d+\eta\in\mathbb{Z}}}\Bigl|\hat{F}_{q^k}\Bigl(\frac{\ell}{d}+\frac{\eta}{q^k}\Bigr)S_{P,q^k}\Bigl(\frac{-\ell}{d}+\frac{-\eta}{q^k}\Bigr)\Bigr|\ll \frac{q^{k/r}(q-1)^k}{\exp(c_q^{1/2}k^{1/2})}.\]
\end{lmm}
\begin{proof}
This follows immediately from Lemma \ref{lmm:LInfBound}, using the trivial bound for the exponential sum involving primes or polynomials.
\end{proof}
\begin{lmm}\label{lmm:PrimeMajor}
Let $A>0$. Then for $D,B<(\log{q^k})^{A}$ and $d>q$ we have
\begin{align*}
\frac{1}{q^k}\sum_{\substack{d<D\\ p|d\Rightarrow p|q}}\sum_{\substack{0\le \ell<d\\ (\ell,d)=1}}\sum_{|b|<B}\hat{F}_{q^k}\Bigl(\frac{\ell}{d}+\frac{b}{q^k}\Bigr)&S_{\Lambda,q^k}\Bigl(\frac{-\ell}{d}+\frac{-b}{q^k}\Bigr)\\
&=\kappa_q(a_0)(q-1)^k+O_A\Bigl(\frac{(q-1)^k}{(\log{q^k})^A}\Bigr),
\end{align*}
where
\[\kappa_q(a_0)=\begin{cases}
\displaystyle\frac{q}{q-1},\qquad&\text{if $(a_0,q)\ne 1$,}\\
\displaystyle\frac{q(\phi(q)-1)}{(q-1)\phi(q)},&\text{if $(a_0,q)=1$.}
\end{cases}
\]
\end{lmm}
\begin{proof}
If $b\ne 0$ then by the prime number theorem in arithmetic progressions in short intervals and partial summation we have
\[S_{\Lambda,q^k}\Bigl(\frac{-\ell}{d}+\frac{-b}{q^k}\Bigr)\ll_A \frac{q^k}{(\log{q^k})^{4A}}.\]
Thus the terms with $b\ne 0$ contribute
\begin{align*}
\ll \frac{(\log{q^k})^{3A}}{q^k}\sup_{0<a<q^k}\Bigl|\hat{F}_{q^k}\Bigl(\frac{a}{q^k}\Bigr)\Bigr|\frac{q^k}{(\log{q^k})^{4A}}\ll \frac{(q-1)^k}{(\log{q^k})^{A}}.
\end{align*}
Here we used the trivial bound that $|\hat{F}_{q^k}(\theta)|\le (q-1)^k$ for all $\theta$.

Using the prime number theorem in arithmetic progressions again, we see that
\[S_{\Lambda,q^k}\Bigl(\frac{-\ell}{d}\Bigr)=\frac{q^k}{\phi(d)}\sum_{\substack{0<c<d\\ (c,d)=1}}e\Bigl(\frac{-l c}{d}\Bigr)+O_A\Bigl(\frac{q^k}{(\log{q^k})^{4A}}\Bigr)=\frac{\mu(d)q^k}{\phi(d)}+O_A\Bigl(\frac{q^k}{(\log{q^k})^{4A}}\Bigr).\]
Thus we may restrict to $d|q$, since all other such $d$ are not square-free. Letting $\ell'/q=\ell/d$, we see terms with $b=0$ and $d|q$ contribute
\begin{align*}
\frac{1}{q^k}\sum_{0\le \ell'<q}\hat{F}_{q^k}\Bigl(\frac{\ell'}{q}\Bigr)&S_{\Lambda,q^k}\Bigl(\frac{-\ell'}{q}\Bigr)=\frac{1}{q^{k-1}}\sum_{\substack{n,m<q^k\\ n\equiv m\Mod{q}}}\Lambda(n)\mathbf{1}_{\mathcal{A}}(m)\\
&=\frac{q}{\phi(q)}\sum_{\substack{1<a<q\\ (a,q)=1}}\sum_{\substack{m<q^k\\ m\equiv a\Mod{q}}}\mathbf{1}_{\mathcal{A}}(m)+O_A\Bigl(\frac{q^k}{(\log{q^k})^{4A}}\Bigr).
\end{align*}
If $a\ne a_0$ then the sum over $m$ is $(q-1)^{k-1}$ since there are $(q-1)$ choices for each digit of $m$ apart from the final one, which must be $a$. If $a=a_0$ then the sum is empty. Thus
%The sum over $c$ is a Ramanujan sum, which (since $(\ell,d)=1$) is equal to $\mu(d)$. In parituclar, there is no contribution from those $d$ which are not square-free, so we may restrict to $d|q$. Thus the expression above is
%\[\sum_{\substack{d|q\\ d>1}}\frac{\mu(d)}{\phi(d)}\sum_{\substack{0<\ell<d\\ (\ell,d)=1}}\hat{F}_{q^k}\Bigl(\frac{\ell}{d}\Bigr)+O\Bigl(\frac{\#\mathcal{A}}{(\log{q^k})^{2A}}\Bigr).\]
%We substitute the definition of $\hat{F}$, so the double sum above is equal to
%\[\sum_{\substack{d|q\\ 1<d<(\log{q^k})^A}}\frac{\mu(d)}{\phi(d)}\sum_{n\le q^k}\mathbf{1}_{\mathcal{A}}(n)\sum_{\substack{0<\ell<d\\ (\ell,d)=1}}e\Bigl(\frac{\ell n}{d}\Bigr).\]
%The inner sum over $\ell$ is also a Ramanujan sum, and so is equal to $\sum_{b|(d,n)}\mu(d/b)b$. Substituting this in above, and swapping the order of summation we obtain
%\[\sum_{\substack{d|q\\ d>1}}\frac{\mu(d)^2}{\phi(d)}\sum_{b|d}\mu(b)b\sum_{\substack{n\le q^k\\ b|n}}\mathbf{1}_{\mathcal{A}}(n).\]
%Since $b|d$ and $d|q$, we have $b|q$, and so the condition $b|n$ only affects the final digit of $n$. If the excluded digit $a_0$ is a multiple of $b$ then there are $q/b-1$ choices for the final digit, whilst if the excluded digit is not a multiple of $b$ there are $q/b$ possible choices of the final digit. There are $(q-1)^{k-1}=\#\mathcal{A}/(q-1)$ possible choices of the remaining digits. Thus
%\[\sum_{\substack{n\le q^k\\ b|n}}\mathbf{1}_{\mathcal{A}}(n)=\frac{q\#\mathcal{A}}{b(q-1)}-\mathbf{1}_{b|a_0}\frac{\#\mathcal{A}}{q-1}.\]
%Substituting this into our expression above, we obtain
\begin{align*}
%&\frac{q\#\mathcal{A}}{q-1}\sum_{\substack{d|q\\ d>1}}\frac{\mu(d)^2}{\phi(d)}\sum_{b|d}\mu(b)-\frac{\#\mathcal{A}}{q-1}\sum_{b|a_0}\mu(b)b\sum_{\substack{d\ne 1\\ d|q,\, b|%d}}\frac{\mu(d)^2}{\phi(d)}\\
%&=-\frac{\#\mathcal{A}}{q-1}\sum_{b|(a_0,q)}\mu(b)\frac{q}{\phi(q)}+\frac{\#\mathcal{A}}{q-1}\sum_{b|a_0}\mu(b)b\mathbf{1}_{b=1}\\
&\frac{q}{\phi(q)}\sum_{\substack{1<a<q\\ (a,q)=1}}\sum_{\substack{m<q^k\\ m\equiv a\Mod{q}}}\mathbf{1}_{\mathcal{A}}(m)=\begin{cases}
q(q-1)^{k-1},\qquad &\text{if $(a_0,q)\ne 1$,}\\
\displaystyle\frac{\phi(q)-1}{\phi(q)}q(q-1)^{k-1},&\text{if $(a_0,q)=1$.}
\end{cases}\qedhere
\end{align*}
\end{proof}
\begin{lmm}\label{lmm:PolyMajor}
For $B,\,q^J<\exp(qk^{1/2})$ we have
\begin{align*}
\frac{1}{q^k}\sum_{\substack{d|q^J}}\sum_{\substack{0\le \ell<d\\ (\ell,d)=1}}\sum_{|b|<B}\hat{F}_{q^k}\Bigl(\frac{\ell}{d}+\frac{b}{q^k}\Bigr)&S_{P,q^k}\Bigl(\frac{-\ell}{d}+\frac{-b}{q^k}\Bigr)\\
&=\mathfrak{S}_J\frac{a_r^{1/r}q^{k/r}(q-1)^k}{q^k}+O\Bigl(\frac{q^{k/r}}{q^{1/2r}}\Bigr),
\end{align*}
where
\[\mathfrak{S}_J=\frac{\#\{(n,m):\,0\le n,m<q^J,\, m\in\mathcal{A},\, P(n)\equiv m\Mod{q^J}\}}{(q-1)^J}.\]
\end{lmm}
\begin{proof}
For $y\ll x^{1-1/2r}$ we have
\begin{align*}
\#\{n:P(n)\in&[x,x+y],n\equiv n_0\Mod{d}\}\\
&=\#\{n:a_r n^r+O(x^{1-1/r})\in[x,x+y], n\equiv n_0\Mod{d}\}\\
&=\frac{y}{d a_r x^{1-1/r}}+O(1).
\end{align*}
Thus the values of $P(n)<q^k$ are well-distributed arithmetic progressions modulo $d<q^J$ and in short intervals of length $\gg q^{k(1-1/2r)}$. Therefore by partial summation we have that for $b\ne 0$ and $d<D$
\[\sum_{P(n)\le q^{k}}e\Bigl(-P(n)\Bigl(\frac{\ell}{d}+\frac{b}{q^k}\Bigr)\Bigr)\ll q^{k/r-1/2r}.\]
In particular, using the trivial bound $|\hat{F}_{q^k}(\theta)|\le (q-1)^k$, we have
\begin{align*}
&\frac{1}{q^k}\sum_{\substack{d|q^J}}\sum_{\substack{0\le \ell<d\\ (\ell,d)=1}}\sum_{0<|b|\ll B}\hat{F}_{q^k}\Bigl(\frac{\ell}{d}+\frac{b}{q^k}\Bigr)\sum_{P(n)\le q^{k}}e\Bigl(-P(n)\Bigl(\frac{\ell}{d}+\frac{b}{q^k}\Bigr)\Bigr)\ll\frac{(q-1)^kq^{k/r}}{q^k q^{1/2r}}.
\end{align*}
Thus we may restrict our attention to $b=0$. Rewriting $\ell/d$ as $\ell'/q^J$ the sum we see that these terms are equal to
\[\frac{1}{q^k}\sum_{0\le \ell'<q^J}\hat{F}_{q^k}\Bigl(\frac{\ell'}{q^J}\Bigr)\sum_{P(n)<q^k}e\Bigl(\frac{-P(n)\ell'}{q^J}\Bigr)=\frac{1}{q^{k-J}}\sum_{m<q^k}\mathbf{1}_{\mathcal{A}}(m)\sum_{\substack{P(n)<q^k\\ p(n)\equiv m\Mod{q^J}}}1.\]
Putting $n$, $m$ into residue classes $\Mod{p^J}$ then gives the result.
\end{proof}
\section{Proof of Theorems \ref{thrm:Prime} and \ref{thrm:Poly}}
\begin{proof}[Proof of Theorem \ref{thrm:Prime}]
By Fourier expansion we have
\begin{align*}
\sum_{n<q^k}\Lambda(n)\mathbf{1}_{\mathcal{A}}(n)&=\frac{1}{q^k}\sum_{0\le a<q^k}\hat{F}_{q^k}\Bigl(\frac{a}{q^k}\Bigr)S_{\Lambda,q^k}\Bigl(\frac{-a}{q^k}\Bigr).
\end{align*}
By Dirichlet's approximation theorem, for any choice of $0<D_0$ and any $0\le a<q^k$ there exists integers $(\ell,d)=1$ with $d<D$ and a real $|\beta|<1/DD_0$ such that
\[\frac{a}{q^k}=\frac{\ell}{d}+\beta.\]
We see that $q^k\ell/d+q^k\beta\in\mathbb{Z}$. We use Lemmas \ref{lmm:PrimeMajor} and \ref{lmm:MajorError} to estimate the contribution when $\max(d,q^k|\beta|)<(\log{q^k})^A$, and use Lemma \ref{lmm:PrimeMinor} for the remaining cases. This gives
\begin{align*}
&\frac{1}{q^k}\sum_{0\le a<q^k}\hat{F}_{q^k}\Bigl(\frac{a}{q^k}\Bigr)S_{\Lambda,q^k}\Bigl(\frac{-a}{q^k}\Bigr)=\kappa_q(a_0)(q-1)^k\\
&+O_A\Bigl( (q-1)^k\Bigl(\frac{1}{(\log{q^k})^A}+\frac{k^4}{(\log{q^k})^{A(1/5-\alpha_q)}}+\frac{k^5 q^{k\alpha_q}}{D_0^{1/2}}+\frac{k^5 D_0^{1/2+2\alpha_q}}{q^{k/2}}\Bigr)\Bigr).
\end{align*}
Choosing $D_0=q^{k/2}$ we see that the error term is $O_B((q-1)^k (\log{q^k})^{-B})$ provided $\alpha_q<1/5$ and $A$ is chosen such that $A>(B+5)/(1/5-\alpha_q)$. We recall from Lemmas \ref{lmm:ExtendedTypeI} and \ref{lmm:L1Bound} that
\[\alpha_q\le\frac{\log\Bigl(\frac{q}{q-1}\log{q}+\frac{3q}{q-1}\Bigr)}{\log{q}}.\]
This clearly tends to zero as $q\rightarrow \infty$. A calculation shows that $\alpha_q<0.198$ for $q>\num{2000000}$. This gives the result.
\end{proof}
\begin{proof}[Proof of Theorem \ref{thrm:Poly}]
The proof is essentially identical to that of Theorem \ref{thrm:Prime} above. We choose $D_0=q^{k/2}$, and split our summation according to $\ell,d,\beta$ such that
\[\frac{a}{a_r r! q^k}=\frac{\ell}{d}+\beta.\]
We use Lemma  \ref{lmm:PolyMinor} in place of \ref{lmm:PrimeMinor} for $\max(d,|\beta|q^k)>q^J$ and Lemma \ref{lmm:PolyMajor} instead of \ref{lmm:PrimeMajor} along with Lemma \ref{lmm:MajorError} to deal with $\max(d,q^k|\beta|)<q^J$. For any choice of $J$ with $q^J<\exp(qk^{1/2})$ we obtain
\begin{align*}
&\frac{1}{q^k}\sum_{0\le a<q^k}\hat{F}_{q^k}\Bigl(\frac{a}{q^k}\Bigr)S_{P,q^k}\Bigl(\frac{-a}{q^k}\Bigr)=\mathfrak{S}_J\frac{a_r^{1/r}q^{k/r}(q-1)^k}{q^k}\\
&+O_A\Bigl(\frac{q^{k/r}(q-1)^k}{q^k}\Bigl(\frac{1}{\exp(c_q^{1/2}k^{1/2})}+\frac{k^4}{q^{J(1/r2^r-\alpha_q)}}+\frac{k q^{k\alpha_q}}{D_0^{1/2^r}}+\frac{D_0^{1/2^r+2\alpha_q}}{q^{k/2^r}}\Bigr)\Bigr).
\end{align*}
Since $D_0=q^{k/2}$, we see that provided $\alpha_q<2^{-r}/r$ the error term is small. In particular there is some quantity $\mathfrak{S}$ such that for any such choice of $J<c_q^{1/2}k^{1/2}$
\[\mathfrak{S}=\mathfrak{S}_J+O\Bigl(\frac{k^4}{q^{J(1/r2^r-\alpha_q)}}\Bigr).\]
Thus, if $\alpha_q<2^{-r}/r$ we see that $\mathfrak{S}_J$ converges to $\mathfrak{S}$ as $J\rightarrow \infty$ and that
\[\frac{1}{q^k}\sum_{0\le a<q^k}\hat{F}_{q^k}\Bigl(\frac{a}{q^k}\Bigr)S_{P,q^k}\Bigl(\frac{-a}{q^k}\Bigr)=\mathfrak{S}\frac{a_r^{1/r}q^{k/r}(q-1)^k}{q^k}+O\Bigl(\frac{q^{k/r}}{\exp(c_q^{1/2}k^{1/2})}\Bigr).\]
Since $\alpha_q\rightarrow 0$ as $q\rightarrow \infty$, we see that $\alpha_q<2^{-r}/r$ for $q>q_0(r)$. From the bound on $C_q\le 1+3/\log{q}$, we see that this holds for $q\ge \exp(\exp(2r))$. %
%
%Finally, we wish to determine when $\mathfrak{S}>0$. If there is a $j$ such that $P(n)$ has one of its final $j$ digits equal to $a_0$ then clearly $\mathfrak{S}=0$. Let $p^u$ be the largest prime power divisor of $P$, and write $P(n)=c+p^j Q(n)$ for some constant $c$, some integer $j$ and some polynomial $Q$ which is not a multiple of $p$. Since $p^u$ is much larger than the degree $r$ of $Q$, we see that for most $n\le p^u$ one has $Q'(u)\not\equiv 0\Mod{p^u}$.
%
%
This completes the proof.
\end{proof}
\section{Modifications for Theorem \ref{thrm:ManyDigits}}
In this section we sketch the modifications required to establish Theorem \ref{thrm:ManyDigits}, leaving the precise details to the interested reader. The results of Section \ref{sec:ExponentialSums} remain unchanged. In Lemma \ref{lmm:L1Bound}, instead of equation \eqref{eq:LittleSum} we have
\[\Bigl|\frac{e(q^{i+1}t)-1}{e(q^i t)-1}-\sum_{i=1}^s e(b_i q^it)\Bigr|\le \min\Bigl(q,s+\frac{1}{2\|q^i t\|}\Bigr).\]
If the $b_i$ are consecutive integers then this can be improved to $\min(2q,1/\|q^it\|)$. Thus we can instead take $C_q=C_{q,s}=1+(2+s)/\log{q}$ in general, or $C_{q,s}=2+2/\log{q}$ if the $b_i$ are consecutive. Lemma \ref{lmm:TypeI} remains unchanged whilst in Lemma \ref{lmm:ExtendedTypeI} all occurrances of $q-1$ should be replaced by $q-s$. In particular, we have
\[\alpha_q=\alpha_{q,s}=\frac{\log\Bigl(C_{q,s}\frac{q}{q-s}\log{q}\Bigr)}{\log{q}}.\]
With these values of $\alpha_{q,s}$ and $C_{q,s}$ in place of $\alpha_q$ and $C_q$, the arguments and statements of Lemmas \ref{lmm:LInfBound}, \ref{lmm:PrimeMinor}, \ref{lmm:PolyMinor}, \ref{lmm:MajorError}, \ref{lmm:PrimeMajor}, \ref{lmm:PolyMajor} all go through as before, except that any occurrence of $q-1$ must be replaced by $q-s$. In Lemma \ref{lmm:MajorError} we made use of the fact that there were two consecutive digits which were not excluded; clearly this still holds in the cases considered by Theorem \ref{thrm:ManyDigits}.

The final proofs of Theorems \ref{thrm:Prime} and \ref{thrm:Poly} then work as before, provided that the constraints $\alpha_{q,s}<1/5$ or $\alpha_{q,s}<r^{-1}2^{-r}$ hold. If the $b_i$ are not neccessarily consecutive then we take $C_{q,s}=1+(2+s)/\log{q}$ and see that if $q$ is sufficiently large in terms of $\epsilon$ and $s<q/2$ then
\[\alpha_{q,s}\le \frac{\log{s}}{\log{q}}+\epsilon.\]
In particular, if $s<q^{1/5-\epsilon}$ then $\alpha_{q,s}<1/5$, as required. A computation reveals that if $q=10^8$ and $s=10$ then $\alpha_{q,s}<1/5$, justifying the remark made after Theorem \ref{thrm:ManyDigits}.

If the $b_i$ are consecutive then we can take $C_{q,s}=2+2/\log{q}$, and see that for $q$ sufficiently large in terms of $\epsilon$ we have
\[\alpha_{q,s}\le \frac{\log{q/(q-s)}}{\log{q}}+\epsilon.\]
Thus $\alpha_{q,s}<1/5$ provided $q-s>q^{4/5+\epsilon}$.
\section{Acknowledgments}
We thank Ben Green for introducing the author to this problem. The author is supported by a Clay research fellowship and a fellowship by examination of Magdalen College, Oxford.
\bibliographystyle{plain}
\bibliography{Digits}
\end{document}